\documentclass[11pt]{amsart}
\usepackage{mathtools, amssymb, amsfonts, amsthm, tikz}
\usepackage{psfrag, color, multicol, graphicx, enumitem, array}
\usepackage{amsfonts, mathdots}
\usepackage{epsfig}
\usepackage{epstopdf}

\newtheorem{thm}{Theorem}[section]

\newtheorem{cor}[thm]{Corollary}
\newtheorem{prop}[thm]{Proposition}

\theoremstyle{plain}
\newtheorem{theo}[thm]{Theorem}
\newtheorem{lem}[thm]{Lemma}

\theoremstyle{definition}
\newtheorem{problem}[equation]{Problem}
\newtheorem{obs}{Observation}

\newtheorem{rem}[thm]{Remark}

\numberwithin{equation}{section}

\def\sq{\square}

\def\zz{\mathbb Z}
\def\nn{\mathbb N}

\def\rr{\mathbb R}

\def\ga{\gamma}

\def\de{\delta}
\def\ep{\epsilon}
\def\eps{\epsilon}
\def\al{\alpha}
\def\be{\beta}

\def\ssu{\subset}

\def\<{\langle}
\def\>{\rangle}

\def\Z{ {\text {\rm Z} } }

\def\rE{{\text {\rm E} } }
\def\Q{{\text {\rm Q} } }

\def\0{{\mathbf 0}}

\def\NN{{\mathbb N}}

\def\.{\hskip.06cm}
\def\ts{\hskip.03cm}

\def\conv{{\text {\rm {conv}} }}

\def\bz{{\textbf{z}}}

\def\poly{\textup{\textsf{P}}}

\def\SP{{\textup{\textsf{\#P}}}}

\def\FPPo{{\textup{\textsf{FP/poly}}}}
\def\SigmaP{\boldsymbol{\Sigma}^{\poly}}
\def\PiP{\boldsymbol{\Pi}^{\poly}}
\def\sharpP{\textup{\textsf{\#P}}}

\def\Z{\mathbb{Z}}
\def\R{\mathbb{R}}

\def\Q{\mathbb{Q}}

\newcommand{\cj}[1]{\overline{#1}}

\renewcommand{\b}{\cj{b}}

\def\a{\cj{a}}

\newcommand{\x}{\mathbf{x}}
\renewcommand{\t}{\mathbf{t}}
\renewcommand{\u}{\mathbf{u}}

\def\aalpha{\boldsymbol{\alpha}}

\newcommand{\y}{\mathbf{y}}
\newcommand{\z}{\mathbf{z}}

\newcommand{\floor}[1]{\lfloor#1\rfloor}

\newcommand{\ex}{\exists\ts}
\renewcommand{\for}{\forall\ts}

\def\nin{\noindent}

\def\NP{{\textup{\textsf{NP}}}}

\newcommand{\cpl}{\ts\backslash\ts}

\def\v{\mathbf{v}}

\def\w{\mathbf{w}}

\newcommand\qrem[1]{\{\!\{ #1 \}\!\}}

\def\fib{\phi}
\def\Fib{\Phi}
\def\chain{{\mathcal C}}
\def\bprime{\cj{b'}}

\renewcommand{\problem}[1]{\textsc{#1}}

\newcommand{\problemdef}[3]{
\bigskip
\begin{tabular}{p{0.1\textwidth} p{0.8\textwidth}}
\multicolumn{2}{l}{\problem{#1}}\\
\textbf{Input:} & #2 \\
\textbf{Decide:} & #3
\end{tabular}
\bigskip
}

\newcommand{\countingdef}[3]{
\bigskip
\begin{tabular}{p{0.1\textwidth} p{0.8\textwidth}}
\multicolumn{2}{l}{\problem{#1}}\\
\textbf{Input:} & #2 \\
\textbf{Output:} & #3
\end{tabular}
\bigskip
}

\def\Boo{\Psi}

\title{The computational complexity of integer programming with alternations}

\author[Danny Nguyen \and Igor Pak]{Danny Nguyen$^{\star}$ \and Igor~Pak$^{\star}$}

\thanks{\thinspace ${\hspace{-.45ex}}^\star$Department of Mathematics,
UCLA, Los Angeles, CA, 90095.
\hskip.06cm
Email:
\hskip.06cm
\texttt{\{ldnguyen,\ts{pak}\}@math.ucla.edu}}

\thanks{
\today}

\begin{document}
\maketitle

\begin{abstract}
We prove that integer programming with three alternating quantifiers is $\NP$-complete, even for a fixed number of variables.
This complements earlier results by Lenstra and Kannan, which together say that integer programming with at most two
alternating quantifiers can be done in polynomial time for a fixed number of variables.  As a byproduct of the proof,
we show that for two polytopes $P,Q \subset \R^{3}$, counting the projection of integer points in $Q \cpl P$ is $\sharpP$-complete.
This contrasts the 2003 result by Barvinok and Woods, which allows counting in polynomial time the projection of
integer points in  $P$ and $Q$ separately.
\end{abstract}

\vskip1.5cm

\section{Introduction}

\subsection{Background}


In a pioneer paper~\cite{L}, Lenstra showed that Integer Programming
in a bounded dimension can be solved in polynomial time.  The next breakthrough
was obtained by Kannan in 1990 and until recently remained the most general
result in this direction (see~\cite{E2}).

\begin{theo}[Parametric Integer Programming~\cite{K1}]\label{th:Kannan}
Fix $d_{1}$ and $d_{2}$.
Given a polyhedron $P \subseteq \R^{d_{1}}$, a matrix $A \in \Z^{m \times (d_{1}+d_{2})}$ and a vector $\b \in \Z^{m}$, the following sentence can be decided in polynomial time:
\begin{equation}\label{eq:Kannan}
\for \x \in P \cap \Z^{d_{1}} \quad \ex \y \in \Z^{d_{2}}  \quad : \quad A \, (\x,\y) \le \b.
\end{equation}
Here $P$ is given by a system $C \, \x \le \cj \ga$, with $C \in \Z^{n \times d_{1}}$ and $\cj \ga \in \Z^{n}$.
The numbers $m,n$ are part of the input.
\end{theo}

In~\cite{K2}, Kannan asked if
Theorem~\ref{th:Kannan} can be extended to three alternating quantifiers.
We give an answer in the negative direction to this question:

\begin{theo}\label{th:main_1}
Fix $d_{1} \ge 1, d_{2} \ge 2$ and $d_{3} \ge 3$.
Given two polyhedra $P \subseteq \R^{d_{1}}$, $Q \subseteq \R^{d_{2}}$, a matrix
$A \in \Z^{m \times (d_{1}+d_{2}+d_{3})}$ and a vector $\b \in \Z^{m}$,
then deciding the sentence
\begin{equation}\label{eq:main_1}
\ex \x \in P \cap \Z^{d_{1}} \quad \for \y \in Q \cap \Z^{d_{2}} \quad \ex \z \in \Z^{d_{3}} \quad : \quad A \, (\x,\y,\z) \le \b
\end{equation}
is an $\NP$-complete problem.
Here $P$ and $Q$ are given by two systems $C \, \x \le \cj\ga$ and $D \, \y \le \cj\de$, with $C \in \Z^{n \times d_{1}}$, $\cj\ga \in \Z^{n}$, $D \in \Z^{q \times d_{2}}$, and $\cj\de \in \Z^{q}$.
\end{theo}

Let us emphasize that in both Theorem~\ref{th:Kannan} and~\ref{th:main_1}, there is no bound on the number of inequalities invovled.
In other words, the parameters $m,n$ and $q$ are \emph{not} fixed.
Theorem~\ref{th:main_1} is especially surprising for the following reasons.
First, in~\cite{short_presburger}, we gave strong evidence that~\eqref{eq:main_1} is decidable in polynomial time if $m,n$ and $q$ are fixed.
Second, by an easy application of the Doignon--Bell--Scarf theorem,~\eqref{eq:Kannan} is polynomial time reducible to the case with $m$ and $n$ fixed.
Unfortunately, this simple reduction breaks down when there are more than two quantifiers (see Section~\ref{ss:finrem-doignon}) as in~\eqref{eq:main_1}.
Still, in~\cite{short_presburger}, we speculated that a more involved reduction argument might still apply to~\eqref{eq:main_1}.
Theorem~\ref{th:main_1} refutes the possibility of any reduction from~\eqref{eq:main_1} to an easier form with $m,n$ and $q$ bounded for which decision could be in polynomial time, unless $\poly = \NP$.
In fact, Theorem~\ref{th:main_1} holds even when $P$ is an interval and
$Q$ is an axis-parallel rectangles
(see Theorem~\ref{th:main_1_restated} and $\S$\ref{ss:finrem-dim}).
Thus, the problem~\eqref{eq:main_1} is already hard when $n,q$ are fixed and only $m$ is unbounded.

\medskip

In~\cite{S}, Sch\"{o}ning proved that it is $\NP$-complete to decide whether
\begin{equation}\label{eq:Schoning}
\ex x \in \Z \quad \for y \in \Z  \quad : \quad \Boo(x,y).
\end{equation}
Compared to~\eqref{eq:main_1}, this has only two quantifiers. However, here the expression $\Boo(x,y)$ is allowed to contain both conjunctions and disjunctions of many inequalities.
So Theorem~\ref{th:main_1} tells us that disjunctions can be discarded at the cost of adding one extra alternation.
In the next subsection, we generalize this observation.

On can also consider a ``hybrid'' version of~\eqref{eq:main_1} and~\eqref{eq:Schoning} with only $2$ quantifiers $\ex\for$ and only $2$ disjunctions in $\Psi$.
In Section~\ref{sec:2_quant}, we show this is still $\NP$-complete to decide.

\subsection{Presburger sentences}

In~\cite{G}, Gr\"{a}del considered the theory of \emph{Presburger Arithmetic},
and proved many completeness results in this theory when the number of variables and quantifiers are bounded.
Those results were later strengthened by Sch\"{o}ning in~\cite{S}.
They can be summed up as follows:

\begin{theo}[\cite{S}]\label{th:Schoning}
Fix $k \ge 1$.
Let $\Boo(\x,\y)$ be a Boolean combination of linear inequalities with integer coefficients in the variables $\x = (x_{1},\dots,x_{k}) \in \Z^{k}$ and $\y = (y_{1},\dots,y_{3}) \in \Z^{3}$.
Then deciding the sentence
\begin{equation*}
Q_{1} \, x_{1} \in \Z  \quad \dots \quad Q_{k} \, x_{k} \in \Z  \quad Q_{k+1} \, \y \in \zz^{3} \quad : \quad \Boo(\x,\y)
\end{equation*}
is $\SigmaP_{k}$-complete if  $Q_{1} = \ex$, and $\PiP_{k}$-complete if $Q_{1} = \for$.
Here $Q_{1},\dots,Q_{k+1} \in \{\for,\ex\}$ are $m+1$ alternating quantifiers.
\end{theo}

This result characterizes the complexity of so called \emph{Presburger sentences} with $k+1$ quantifiers in a fixed number of variables.
The main difference between Presburger Arithmetic versus integer programming is that the expression $\Boo$ allows both conjunction and disjunction of many inequalities.
This flexibility allows effective reductions of classical decision problems such as $\textsc{QSAT}$.
For some time, it remains a question whether such reductions can be carried with only conjunctions, and at the same time keeping the number of variables fixed.
We prove the following result, which generalizes Theorem~\ref{th:main_1}:

\begin{theo}\label{th:main_2}
Integer programming in a fixed number of variables with $k+2$ alternating quantifiers
is $\SigmaP_{k}/\PiP_{k}$-complete, depending on whether $Q_{1} = \ex/\for$.
Here the problem is allowed to contain only a system of inequalities.
\end{theo}

We refer to Theorem~\ref{th:main_2_restated} for the precise statement.
Thus, we see that integer programming requires only one more quantifier alternation to achieve the same complexity as Presburger Arithmetic.
Again, we emphasize that while the number of variables and quantifiers are fixed in Theorem~\ref{th:main_2}, the linear system is still allowed many inequalities.


\subsection{Counting points in projections of non-convex polyhedra}
For polytopes in arbitrary dimension, counting the number of integer points points is classically $\sharpP$-complete,
even for 0/1~polytopes.  In a fixed dimension~$d$, Barvinok famously showed this can be done in
polynomial time:

\begin{theo}[\cite{B1}]\label{th:B}
Fix $d$. Given a polytope $P \subset \R^{d}$, the number of integer points in $P \cap \Z^{d}$ can be computed in polynomial time.
Here $P$ is described by a system $A\x \le \b$, with $A \in \Z^{m\times d},\, \b \in \Z^{m}$.
\end{theo}

For a set $S \subset \R^{d}$, denote by $\rE(S):= S \cap \Z^{d}$.
The previous results say that $|\rE(P)|$ is computable in polynomial time.
Given two polytopes $P \subset Q \subset \R^{d}$, we clearly have $|\rE(Q \cpl P)| = |\rE(Q)| - |\rE(P)|$.
So the number of integer points in a complement can also be computed effectively.

\medskip

Theorem~\ref{th:B} was later generalized by Barvinok and Woods to count the
number of integer points in projections of polytopes:

\begin{theo}[\cite{BW}]\label{th:BW}
Fix $d_{1}$ and $d_{2}$.
Given a polytope $P \subset \R^{d_{1}}$, and a linear transformation $T : \Z^{d_{1}} \to \Z^{d_{2}}$,
the number of integer points in $T(P \cap \Z^{d_{1}})$ can be computed in polynomial time.
Here $P$ is described by a system $A\x \le \b$ and $T$ is described by a matrix $M$,
where $A \in \Z^{m\times d_{1}},\, \b \in \Z^{m}$ and $M \in \Z^{d_{2} \times d_{1}}$.
\end{theo}

For a set $S \subset \R^{d}$, denote by $\rE_{1}(S)$ the projection of $S \cap \Z^{d}$ on the first coordinate, i.e.,
\begin{equation*}
\rE_{1}(S) \, := \, \{ x \in \Z \quad : \quad \ex \z \in \Z^{d-1} \quad (x,\z) \in S\}.
\end{equation*}
By Theorem~\ref{th:BW}, $|\rE_{1}(P)|$ can be computed in polynomial time for every polytope $P \subset \R^{d}$.

\medskip

We prove the following result:

\begin{theo}\label{th:main_3}
Given two polytopes $P \subset Q \subset \R^3$, computing $|\rE_{1}(Q \cpl P)|$ is $\sharpP$-complete.
\end{theo}

In other words, it is $\sharpP$-complete to compute the size of the set
\begin{equation}\label{eq:main_3}
\rE_{1}(Q \cpl P) \. = \. \{x \in \Z \; : \; \ex \z \in \Z^{2} \quad (x,\z) \in Q \cpl P\}\ts.
\end{equation}
Note that the corresponding decision problem $|\rE_{1}(Q \cpl P)|\ge 1$ is equivalent to
$|\rE(Q \cpl P)|\ge 1$, and thus can be decided in polynomial time by applying Theorem~\ref{th:B}.

\medskip

The contrast between Theorem~\ref{th:BW} and our negative result can be explained as follows.
The proof Theorem~\ref{th:BW} depends on the polytopal structure of $P$ and exploited
convexity in a crucial way.  By taking the complement $Q \cpl P$, we no longer have a convex set.
In other words, we show that projection of the complement $Q \cpl P$ is complicated enough to
allow encoding of hard counting problems, even in $\rr^{3}$ (see also $\S$\ref{ss:finrem-diff}).

\begin{figure}[hbt]
\begin{center}
\psfrag{P}{$P$}
\psfrag{Q}{$Q$}
\epsfig{file=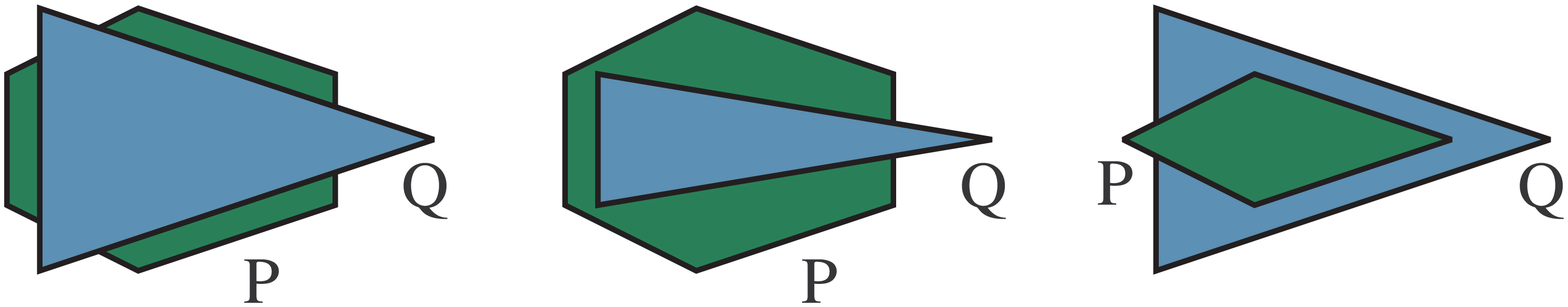, width=11.2cm}
\vskip-.15cm
\end{center}
\caption{Three examples of convex polygons $P,Q\ssu \rr^2$. }
\label{f:diff}
\end{figure}

\begin{rem}\label{rem:diff}
To understand the theorem, consider three examples of polygons $P,Q\ssu \rr^2$ as in Figure~\ref{f:diff}.
Note that the sets of integer points of the vertical projections of $P,Q$ and $P\cup Q$ are the same in all three cases,
but the sets number of integer points of the vertical projections of $Q\cpl P$ are quite different.
\end{rem}



As an easy consequence of Theorem~\ref{th:main_3} we obtain:

\begin{cor}\label{th:many_proj}
Given $r$ simplices $T_{1},\dots,T_{r} \subset \R^{3}$, computing $|\rE_{1}(T_{1} \cup\dots\cup T_{r})|$ is $\sharpP$-complete.
\end{cor}

\medskip

\subsection{Outline of the paper}

We begin with notations (Section~\ref{sec:notation}) and a geometric
construction of certain polytopes based on Fibonacci numbers
(Section~\ref{sec:geom-constr}).  In Section~\ref{sec:main_1_proof} we
use this construction to prove Theorem~\ref{th:main_1} via a reduction of the
\textsc{GOOD SIMULTANEOUS APPROXIMATION (GSA)} Problem in Number Theory,
which is known to be $\NP$-complete.
The proof of Theorem~\ref{th:main_2} is via a reduction of $\textsc{QSAT}$ (Section~\ref{sec:main_2_proof}).
The proof of Theorem~\ref{th:main_3}
follows a similar route via reduction of
\textsc{\#{}GSA} (Section~\ref{sec:main_3_proof}).
Then we show that a ``hybrid'' version of~\eqref{eq:main_1} and~\eqref{eq:Schoning} with only $2$ quantifiers and $2$ disjunctions is still $\NP$-complete to decide (Section~\ref{sec:2_quant}).
Finally, we conclude with final remarks and open problems (Section~\ref{s:finrem}).

\bigskip

\section{Notations} \label{sec:notation}

\nin
We use $\nn \ts = \ts \{0,1,2,\ldots\}$ and $\Z_{+} = \{1,2,\ldots\}$.

\nin
All constant vectors are denoted $\a, \b, \cj x, \cj y, \cj t$ etc.


\nin
Matrices are denoted $A, B, C$, etc.

\nin
Variables are denoted $x,y,z$, etc.; vectors of variables are denoted $\x, \y, \z$, etc.

\nin
We write $\x \le \y$ if $x_{j} \le y_{j}$ for all $i$.











\nin
A \emph{polyhedron} is an intersection of finitely many closed half-spaces in~$\rr^n$.



\nin
A \emph{polytope} is a bounded polyhedron.

\nin
Polyhedra and polytopes are denoted by $P, Q, R$, etc.



\bigskip

\section{Geometric constructions and properties}  \label{sec:geom-constr}

\subsection{Fibonacci points}\label{sec:Fibonacci}

We consider the first $2d$ \emph{Fibonacci numbers}:
\begin{equation*}
F_{0} = 0, F_{1} = 1, F_{2} = 1, \dots, F_{2d-1}.
\end{equation*}
From these, we construct $d$ integer points:
\begin{equation}\label{eq:y_i_def}
\fib_{1} = (F_{1},F_{0}),\; \fib_{2} = (F_{3},F_{2}),\; \dots,\; \fib_{d} = (F_{2d-1},F_{2d-2}).
\end{equation}
Let
\begin{equation}\label{eq:Fib_J_def}
\Fib = \{\fib_{1}, \dots, \fib_{d} \} \subset\Z^{2}  \quad \text{and} \quad J = [1, F_{2d-1}] \times [0,F_{2d-2}] \cap \Z^{2}.
\end{equation}
We have $\Fib \subset J$.
Denote by $\chain$ the curve consisting of $d-1$ segments connecting $\fib_{i}$ to $\fib_{i+1}$ for $i=1,\dots,i-1$.

We also define the following two polygons.
Their properties will be mentioned later.
\begin{equation}\label{eq:part_above}
R_{1} =
\Biggl\{ \y = (y_{1},y_{2}) \in \R^{2} \; : \;
\left[
\begin{smallmatrix}
y_{1} &\ge &1\\
y_{2} &\le &F_{2d-2} \\
y_{2}F_{2d-1} - y_{1}F_{2d-2} &\ge &1
\end{smallmatrix}
\right]
\Biggr\},
\end{equation}
and
\begin{equation}\label{eq:part_below}
R_{2} =
\Biggl\{ \y \in \R^{2} \; : \;
\left[
\begin{smallmatrix}
y_{1} &\le &F_{2d-1}\\
y_{2} &\ge &0 
\end{smallmatrix}
\right]
\text{ and } \; y_{2} F_{2i} - y_{1} F_{2i-1} \le -2 \;
\text{ for } \;
i =1,\dots,d
\Biggr\}.
\end{equation}



\medskip

\nin The following properties are straightforward from the above definitions:

\medskip

\begin{enumerate}[label=(F{\arabic*})]
\setlength\itemsep{0.5em}

\item The points $\fib_{1}, \dots, \fib_{d}$ are in convex position.
The curve $\chain$ connecting them is convex (upwards).
See Figure~\ref{fig:Fib}.

\item Each segment $(\fib_{i}\. \fib_{i+1})$ and each triangle $\Delta_{i} = (0 \. \fib_{i} \. \fib_{i+1})$ has no interior integer points.
This can be deduced from the facts that two consecutive Fibonacci numbers are coprime, and also
\begin{equation*}
F_{i}F_{i+3} - F_{i+1}F_{i+2} = (-1)^{i-1} \quad \text{for all } i \ge 0.
\end{equation*}

\item The set of integer points in $J \cpl \Fib$ can be partitioned into $2$ parts: those lying strictly above the convex curve $\chain$, and those lying strictly below it.

\item The part of $J \cpl \Fib$ lying above $\chain$ is exactly $R_{1} \cap \Z^{2}$.
This can be seen as follows.
The line $\ell$ connecting $0$ and $\fib_{d}$ is defined by:
\begin{equation*}
y_{2} F_{2d-1} - y_{1} F_{2d-2}  = 0.
\end{equation*}
So every integer point $\y = (y_{1},y_{2})$ lying above $\ell$ satisfies:
\begin{equation*}
y_{2} F_{2d-1} - y_{1} F_{2d-2} \ge 1.
\end{equation*}
By property (F2), there are no integer points $\y$ between $\chain$ and $\ell$.
The other two edges of $R_{1}$ come from $J$.
See Figure~\ref{fig:Fib}.

\item The part of $J \cpl \Fib$ lying below $\chain$ is exactly $R_{2} \cap \Z^{2}$.
This can be seen as follows.
The line connecting $\fib_{i}$ and $\fib_{i+1}$ is defined by
\begin{equation*}
y_{2} F_{2i} - y_{1} F_{2i - 1} = -1.
\end{equation*}
So all integer points below that line satisfies:
\begin{equation*}
y_{2} F_{2i} - y_{1} F_{2i - 1} \le -2.
\end{equation*}
This gives $d-1$ faces for $R_{2}$, one for each $1 \le i \le d-1$.
The other two faces of $R_{2}$ come from from $J$.
See Figure~\ref{fig:Fib}.

\end{enumerate}

\begin{figure}[!h]
\centering
\begin{tikzpicture}[scale=1.5]
\draw[->] (0.2,0)--(4.5,0);
\draw[->] (0.2,0)--(0.2,3);

\coordinate (A) at (0.85,0);
\coordinate (B) at (1.2,0.5);
\coordinate (C) at (2.0,1.3);
\coordinate (D) at (4.1,3);
\coordinate (preD) at (3.35,2.4);

\coordinate (E) at (0.85,0.62);
\coordinate (F) at (0.85,3);

\coordinate (A') at (1.15,0);
\coordinate (B') at (1.5,0.5);
\coordinate (C') at (2.215,1.2);
\coordinate (D') at (4.1,2.8);
\coordinate (preD') at (3.5,2.3);
\coordinate (D'') at (3.95,3);


\draw[black,fill=black] (A) circle (.2ex);
\draw[black,fill=black] (B) circle (.2ex);
\draw[black,fill=black] (C) circle (.2ex);
\draw[black,fill=black] (D) circle (.2ex);

\draw[red, line width = 0.6] (A')--(B')--(C');
\draw[red, line width = 0.6] (preD')--(D')--(4.1,0)--(A');
\draw[red, line width = 0.6] (D'')--(E)--(F)--cycle;
\draw[-,blue, line width = 1] (A)--(B)--(C);
\draw[-,blue, line width = 1] (preD)--(D);

 \node (R1) at (1.9,2.2) {$R_{1}$};
 \node (R2) at (3.2,0.9) {$R_{2}$};

 \node (O) at (0,0) {$0$};
 \node (y2) at (0,3) {$y_{2}$};
 \node[right] (y1) at (4.5,0) {$y_{1}$};

 \node[right] (fibd) at (D) {$\fib_d$};
 \node[above left] (fib1) at (A) {$\fib_1$};

 \node (dots) at (2.8,1.9) {$\iddots$};

\end{tikzpicture}
\caption{The points $\fib_{1}, \dots, \fib_{d} \in \Fib$ form a convex curve $\chain$ (blue).}
\label{fig:Fib}%
\end{figure}

\subsection{The polytopes}\label{sec:polytopes}

Given $\aalpha = (\al_{1}, \dots, \al_{d}) \in \Q^{d}$ and $\eps \in (0,\frac{1}{2}) \cap \Q$, for each $1 \le i \le d$, we define a polygon:
\begin{equation}\label{eq:P_i_def}
P_{i} = \big\{(x,w) \in \R^{2} \quad : \quad 1 \le x \le N,\quad  \al_{i} x - \eps \le w \le  \al_{i} x + \eps \big\}.
\end{equation}
Next, for each $1 \le i \le d$, we define a new polygon
\begin{equation}\label{eq:P'_i_def}
P'_{i} = \big\{ (x,\fib_{i},w) \;:\;  (x,w) \in P_{i} \big\} \subset \R^{4}.
\end{equation}
Finally, we define the convex hull:
\begin{equation}\label{eq:P_def}
P = \conv(P'_{1}, \dots, P'_{d}) \subset \R^{4}.
\end{equation}

\nin The following properties are straightforward from the above definitions:

\medskip

\begin{enumerate}[label=(P{\arabic*})]
\setlength\itemsep{0.5em}

\item Each $P_{i}$ is a parallelogram with vertices $\big\{ (1, \al_{i} \pm \eps),\, (N, \al_{i} N \pm \eps) \big\}$.


\item Each $P'_{i}$ is a \emph{parallelogram} in $\R^{4}$ (i.e., a Minkowski sum of two intervals), with vertices $\big\{ (1, \fib_{i}, \al_{i} \pm \eps),\, (N, \fib_{i}, \al_{i} N \pm \eps) \big\}$.



\item The set of all vertices from $P'_{1},\dots,P'_{d}$ are in convex position.
Each $P'_{i}$ forms a $2$-dimensional face of $P$.
This follows from from~\eqref{eq:P'_i_def} and (F1).

\item The polytope $P$ has $4d$ vertices, which are all the vertices of $P'_{1},\dots, P'_{d}$.

\item
For every vertex $(x,\y,w)$ of $P$, we have $\y = \fib_{i} \in \Fib$ for some $1 \le i \le d$.
Conversely, for every $\fib_{i} \in \Fib$, we have:
\begin{equation*}
\big\{(x,w) \in \R^{2} : (x,\fib_{i},w) \in P \big\} = P_{i}.
\end{equation*}




\end{enumerate}

\nin We will be using these properties in the latter sections.

\bigskip

\section{Proof of Theorem~\ref{th:main_1}}\label{sec:main_1_proof}
\subsection{}
By a \emph{box} in $\Z^{d}$, we mean the set of integer points of the form $[\al_{1},\be_{1}]\times\dots\times[\al_{d},\be_{d}] \cap \Z^{d}$.
We will prove the following stronger version of Theorem~\ref{th:main_1}.

\begin{theo}\label{th:main_1_restated}
Given a polytope $U \subset \R^{6}$ and two finite boxes $I \subset \Z$, $J \subset \Z^{2}$, deciding the sentence
\begin{equation}\label{eq:main_1_restated}
\ex x \in I \quad \for \y \in J \quad \ex \z \in \Z^{3} \quad : \quad (x,\y,\z) \in U
\end{equation}
is an $\NP$-complete problem.
Here $U$ is described by a system $A \, (x,\y,\z) \le \b$, where $A \in \Z^{m \times 6}$ and $\b \in \Z^{m}$.
\end{theo}

Since low dimensional boxes can be easily embedded into higher dimensions,
the above implies Theorem~\ref{th:main_1} for every $d_{1}\ge 1, d_{2} \ge 3$ and $d_{3} \ge 3$.
Compared to Theorem~\ref{th:main_1}, all parameters in the above theorem are fixed, except for $m$.
So from now on, the symbols $n$ and $d$ will be reused for other purposes.
For a vector $\aalpha = (\al_{1},\dots,\al_{d}) \in \Q^{d}$ and an integer $x \in \Z$, we define
\begin{equation}\label{eq:qrem_def}
\qrem{x\aalpha} = \max_{1 \le i \le d} \qrem{q\al_{i}},
\end{equation}
where for each rational $\be \in \Q$, the quantity $\{\be\}$ is defined as:
\begin{equation*}
\qrem{\be} \. := \. \min_{n\in\Z} | \be - n | \. = \. \min \ts \bigl\{ \be - \floor{\be}, \lceil{\be}\rceil - \be \bigr\}\ts.
\end{equation*}

\problemdef{GOOD SIMULTANEOUS APPROXIMATION (GSA)}
{A rational vector $\aalpha = (\al_{1}, \dots, \al_{d}) \in \Q^{d}$ and $N \in \NN$, $\eps \in \Q$.}
{Is an integer $x \in [1, N]$ such that $\qrem{x\aalpha} \le \eps\ts$?}

Note that $\textsc{GSA}$ is only non-trivial for $\eps < 1/2$.
\nin We need the following result by Lagarias:

\begin{theo}[\cite{Lag}]
$\textsc{GSA}$ is $\NP$-complete.
\end{theo}

Let us emphasize that in $\textsc{GSA}$, the number $d$ is part of the input.
If $d$ is fixed instead, then the problem can be decided in polynomial time (see~\cite{Lag} and~\cite[Ch.~5]{GLS}).
What follows is a reduction of $\textsc{GSA}$ to a sentence of the form~\eqref{eq:main_1_restated}.
$\textsc{GSA}$ can be expressed as an integer programming problem:
\begin{equation}\label{eq:GSAIP}
\ex \, x, w_{1}, \dots, w_{d}  \in \Z \quad : \quad 1 \le x \le N,\quad -\ep \le   \al_{i} x - w_{i} \le \ep.
\end{equation}
The inequalities on $w_{i}$ can be expressed as $(x,w_{i}) \in P_{i}$, where $P_{i}$ was defined in~\eqref{eq:P_i_def}.
Letting $I = [1,N] \cap \Z$, we see that $\textsc{GSA}$ is equivalent to deciding:
\begin{equation}\label{eq:all_P_i}
\ex x \in I \quad : \quad \bigwedge_{i=1}^{d} \Big[ \ex w \in \Z \;:\; (x,w) \in P_{i} \Big].
\end{equation}

\begin{lem} Let $\Fib = \{\fib_{1},\dots, \fib_{d}\}$ be as in~\eqref{eq:Fib_J_def} and $P$ be as in~\eqref{eq:P_def}. We have:
\begin{equation}\label{eq:intermediate}
\qrem{x\aalpha} \le \eps \quad \iff \quad \for \y \in \Fib \quad \ex w \in \Z : (x,\y,w) \in P.
\end{equation}
\end{lem}

\begin{proof}
Indeed, assume $\qrem{x\aalpha} \le \eps$, i.e., $x$ satisfies $\textsc{GSA}$.
By~\eqref{eq:all_P_i}, for every $i = 1,\dots,d$, there exists $w_{i} \in \Z$ with $(x,w_{i}) \in P_{i}$.
Now (P5) implies that $(x,\fib_{i},w_{i}) \in P$.
Since this holds for every $\fib_{i} \in \Fib$, the RHS in~\eqref{eq:intermediate} is satisfied.
For the other direction, assume the RHS in~\eqref{eq:intermediate} holds.
Then for every $\fib_{i}  \in \Fib$, there exists $w_{i} \in \zz$ with $(x,\fib_{i},w_{i}) \in P$.
By (P5), we have $(x,w_{i}) \in P_{i}$.
By~\eqref{eq:all_P_i}, $x$ satisfies $\textsc{GSA}$, i.e., $\qrem{x\aalpha} \le \eps$.
\end{proof}

\medskip

By the above lemma, $\textsc{GSA}$ is equivalent to:
\begin{equation}\label{eq:intermediate_2}
\ex x \in I \quad \for \y \in \Fib \quad   \ex w \in \Z : (x,\y,w) \in P .
\end{equation}
Consider $J$ from~\eqref{eq:Fib_J_def}, which contains $\Fib$.
We can rewrite the above sentence as:
\begin{equation}\label{eq:complement}
\ex x \in I \quad \for \y \in J \quad \big[  (\y \in J \cpl \Fib) \; \lor \; \ex w \in \Z : (x,\y,w) \in P \big].
\end{equation}
Recall the polygons $R_{1}$ and $R_{2}$ defined in~\eqref{eq:part_above} and~\eqref{eq:part_below}.
By properties (F3), (F4) and (F5), we can rewrite $\y \in J \cpl \Fib$ as $(\y \in R_{1}) \lor (\y \in R_{2})$.
Now, we can rewrite~\eqref{eq:complement} as:
\begin{equation}\label{eq:step_2}
\ex x \in I \quad \for \y \in J \quad \big[ \, (\y \in R_{1}) \;\lor\; (\y \in R_{2}) \;\lor\; \ex w \in \Z : (x,\y,w) \in P \, \big].
\end{equation}

\medskip

Next, define two polytopes $R'_{1}$ and $R'_{2}$ as follows:
\begin{equation}\label{eq:R'_i_def}
R'_{i} \. := \ts \bigl\{ (x,\y,0) \in \R^{4} : 0 \le x \le N,\; \y \in R_{i} \bigr\} \ts \subset \R^{4} \quad  \text{for} \quad  i = 1,2.
\end{equation}
Polytopes $R'_{1}$ and $R'_{2}$ are defined in such a way so that for every $x \in I$ and $\y \in J$,
we have $\y \in R_{i}$ if and only if there exists $w\in\Z$ such that
$(x,\y,w) \in R'_{i}$.\footnote{Such a $w$ must automatically be $0$ by the definition of $R'_{i}$.}
Now, it is clear that~\eqref{eq:step_2} is equivalent to:
\begin{equation*}
\ex x \in I \quad \for \y \in J \quad \Bigg[ \Biggl( \, \bigvee_{i=1}^{2} \ex w \in \Z : (x,\y,w) \in R'_{i}  \, \Biggr) \;\lor\; \Biggl( \ex w \in \Z : (x,\y,w) \in P \Biggr) \, \Bigg].
\end{equation*}
which is equivalent to:
\begin{equation}\label{eq:step_3}
\ex x \in I \quad \for \y \in J \quad \ex w \in \Z \quad : \quad (x,\y,w) \in R'_{1} \cup R'_{2} \cup P.
\end{equation}
The difference between~\eqref{eq:step_3} and~\eqref{eq:main_1_restated} is that we have $3$ polytopes instead of just one.

\medskip

\subsection{}
The final step is two compress three polytopes $R'_{1},R'_{2}$ and $P$ into one polytope.
Recall from (P4) that $P$ has $4d$ vertices, which correspond to the vertices of all $P_{i}$ for $1 \le i \le d$.
The vertices of $R_{1}$ and $R_{2}$ can be computed in polynomial time from systems~\eqref{eq:part_above} and~\eqref{eq:part_below}.
From there we easily get the vertices of $R'_{1}$ and $R'_{2}$.
Since $P, R'_{1}$ and $R'_{2}$ are in the fixed dimension $4$, we can write down all their facets in polynomial time using their vertices.
So we can represent:
\begin{equation}\label{eq:systems}
\aligned
P &= \left\{(x,\y,w) \in \R^{4} \; : \; A_{1} \, (x,\y,w) \le \cj b_{1} \right\},\\
R'_{1} &= \left\{(x,\y,w) \in \R^{4} \; : \; A_{2} \, (x,\y,w) \le \cj b_{2} \right\},\\
R'_{2} &= \left\{(x,\y,w) \in \R^{4} \; : \; A_{3} \, (x,\y,w) \le \cj b_{3} \right\}.
\endaligned
\end{equation}
The above three systems all have lengths polynomial in the input $\aalpha, N$ and $\eps$.
Next, we need the following lemma:

\begin{lem}\label{lem:compress}
Fix $n$ and $r$. Given $r$ polytopes $R_{1},\dots,R_{r} \subset \R^{n}$ described by $r$ systems
\begin{equation*}
R_{i} = \{ \x \in \R^{n} : A_{i} \, \x \le \cj b_{i} \},
\end{equation*}
there is a polytope $U \in \R^{n+\ell}$, where $\ell = \lceil \log_{2} r \rceil$, such that
\begin{equation}\label{eq:union}
\x \in \bigcup_{i=1}^{r} R_{i} \cap \Z^{n} \quad \iff \quad \ex \t \in \Z^{\ell} : (\x,\t) \in U \cap \Z^{n+\ell}.
\end{equation}
Furthermore, the system $A \, (\x,\t) \le \b$ that describes $U$ can be found in polynomial time, given $A_{i}$'s and $\cj b_{i}$'s as input.
\end{lem}

\begin{proof}
Let $\ell = \lceil \log_{2} r \rceil$, we have $2^{\ell} \ge r$. Pick $\cj t_{1}, \dots, \cj t_{r} \in \{0,1\}^{\ell}$ as $r$ different vertices of the $\ell$-dimensional unit cube.
Define
\[
U_{j} = \{(\x, \cj t_{j}) \in \R^{n+\ell} : \x \in R_{j}\} \quad \text{for} \quad j = 1,\dots,r \, ,
\]
and
\begin{equation*}
U = \conv(U_{1}, \dots, U_{r}).
\end{equation*}
In other words, we form $U_{j}$ by augmenting each $R_{j}$ with $\ell$ coordinates of $\cj t_{j}$.
Since $\cj t_{1},\dots,\cj t_{r}$ are in convex position, so are the new polytopes $U_{1},\dots,U_{j}$.
So the vertices of $U$ are all the vertices of all $U_{j}$.
Note that for every $\t \in \conv(\cj t_{1}, \dots, \cj t_{r})$, we have $\t \in \Z^{\ell}$ if and only if $\t = \cj t_{j}$ for some $j$.
This implies that the only integer points in $U$ are those in $U_{j}$'s.
In other words:
\begin{equation*}
(\x, \t) \in U \cap \Z^{n+\ell} \quad \iff \quad \x \in R_{j} \cap \Z^{n} \; \text{ and } \; \t = \cj t_{j} \; \text{ for some } \; j = 1,\dots,r.
\end{equation*}
So we have~\eqref{eq:union}.

For each $R_{j}$, its vertices can be computed in polynomial time from the system $A_{i} \,  \x \le \cj b_{i}$.
From these, we easily get the vertices for each $U_{j}$.
Thus, we can find all vertices of $U$ in polynomial time.
Note that $U$ is in a fixed dimesion $n + \ell$, since $n$ and $r$ are fixed.
Therefore, we can find in polynomial time all the facets of $U$  using those vertices.
This gives us a system $A \, (\x,\t) \le \cj b$ of polynomial length that describes $U$.
\end{proof}

Applying the above lemma for three polytopes $R'_{1},R'_{2}$ and $P$ with $n=4$ and $r=3$, we find a polytope $U \subset \R^{4+\ell}$ such that:
\begin{equation}\label{eq:apply_lem}
(x,\y,w) \in (R'_{1} \cup R'_{2} \cup P) \cap \Z^{4} \quad \iff \quad \ex \t \in \Z^{\ell} : (x,\y,w,\t) \in U \cap \Z^{4 + \ell}.
\end{equation}
Here we have $\ell = \lceil \log_{2}3 \rceil = 2$, which means $\t \in \Z^{2}$ and $U \subset \R^{6}$.
The lemma also allows us to find a system $A \, (x,\y,w,\t) \le \cj b$ that describes $U$, which has size polynomial in the systems in~\eqref{eq:systems}.
Now, we can rewrite~\eqref{eq:step_3} as:
\begin{equation*}
\ex x \in I \quad \for \y \in J \quad \ex w \in \Z \quad : \quad \ex \t \in \Z^{2} \quad (x,\y,w,\t) \in U,
\end{equation*}
which is equivalent to
\begin{equation*}
\ex x \in I \quad \for \y \in J \quad \ex \z \in \Z^{3} \quad : \quad A \, (x,\y,\z) \le \cj b.
\end{equation*}
Here $\z = (w,\t) \in \zz^{3}$.
The final system $A \, (x,\y,\z) \le \cj b$ still has size polynomial in the original input $\aalpha, N$ and $\eps$.
Therefore, the original $\textsc{GSA}$ problem is equivalent to~\eqref{eq:main_1_restated}.
This implies that~\eqref{eq:main_1_restated} is $\NP$-hard.

\medskip

It remains to show that~\eqref{eq:main_1_restated} is in $\NP$.
We argue that  more general sentence~\eqref{eq:main_1} is also in $\NP$.
From a result in~\cite{G}, if~\eqref{eq:main_1} is true, there must be an $\x$ satisfying it with length polynomial in the input $P, A$ and $\b$.
For such an $\x$, we can apply Theorem~\ref{th:Kannan} to check the rest of the sentence, which has the form $\for \y \ex \z$, in polynomial time.
This shows that deciding~\eqref{eq:main_1} is in $\NP$, and thus $\NP$-complete. \ $\sq$

\bigskip

\section{Proof of Theorem~\ref{th:main_2}} \label{sec:main_2_proof}

Recall the definition of boxes from Section~\ref{sec:main_1_proof}.
In this section, we prove:

\begin{theo}\label{th:main_2_restated}
Fix $k \ge 1$.
Given a polytope $U \subset \R^{k+7} $ and finite boxes $I_{1},\dots,I_{k} \subset \Z$, $J \subset \Z^{2}$, $K \subset \Z^{5}$, then the problem of deciding:
\begin{equation}\label{eq:QSAT_converted}
Q_{1} \, x_{1} \in I_{1} \quad \dots \quad Q_{k} \, x_{k} \in I_{k} \, \quad \for \y \in J \quad \ex \z \in K \quad : \quad (\x,\y,\z) \in U
\end{equation}
is $\SigmaP_{k}$ complete if $Q_{1} = \ex$, and $\PiP_{k}$ complete if $Q_{1} = \for$.
Here $Q_{1},\dots,Q_{k}\in\{\ex,\for\}$ are $k$ alternating quantifiers with $Q_{k} = \ex$.
The polytope $U$ is described by a system $A \, (\x, \y, \z) \le \cj b$, where $A \in \Z^{m\times(k+7)}$ and $\b \in \Z^{m}$.
\end{theo}

For the proof, we work with the canonical problem $\textsc{Q3SAT}$.
Let $\Boo$ a Boolean expression of the form:
\setlength\itemsep{0.5em}
\begin{equation}\label{eq:Boo}
\Boo (\u_{1},\dots,\u_{k})  =  \bigwedge_{i=1}^{N} (a_{i} \lor b_{i} \lor c_{i}).
\end{equation}
Here each $\u_{j} = (u_{j1}, \dots, u_{j\ell}) \in \{0,1\}^{\ell}$ is a tuple of $\ts\ell\ts$ Boolean variables, and each $a_{i}, b_{i}, c_{i}$ is a literal in the set
$\; \{ u_{js},\; \lnot u_{js} \; : \; 1 \le j \le k,\; 1 \le s \le \ell \}$.
From $\Boo$, we construct a sentence:
\begin{equation}\label{eq:QSAT}
Q_{1} \, \u_{1} \in \{0,1\}^{\ell} \quad Q_{2} \, \u_{2} \in \{0,1\}^{\ell}  \; \dots \; Q_{k} \, \u_{k} \in \{0,1\}^{\ell}  \quad : \quad \Boo(\u_{1},\dots,\u_{k}).
\end{equation}
Here $Q_{1},Q_{2},\dots,Q_{k} \in \{\for,\ex\}$ are $k$ alternating quantifiers with $Q_{k} = \ex$.
The numbers $\ell$ and $N$ are part of the input.

\problemdef
{QUANTIFIED 3-SATISFIABILITY (Q3SAT)}
{A Boolean expression $\Boo$ of the form~\eqref{eq:Boo}.}
{The truth of the sentence~\eqref{eq:QSAT}.}

For clarity, we use the notation $\textsc{Q3SAT}_{k}$ to emphasize problem~\eqref{eq:QSAT} for a fixed $k$.
It is well-known that $\textsc{Q3SAT}_{k}$ is $\SigmaP_{k}$-complete if $Q_{1} = \ex$ and $\PiP_{k}$-complete if $Q_{1} = \for$ (see
e.g.~\cite{Pap,MM} and~\cite{AB}).
We proceed to reduce~\eqref{eq:QSAT} to~\eqref{eq:QSAT_converted}.
In fact, by representing each Boolean string $\u_{j} \in \{0,1\}^{\ell}$ as an integer $x_{j} \in [0,2^{\ell})$, we will only need to use $I_{1} = I_{2} = \dots = I_{k} = [0,2^{\ell}) \cap \Z$.

\medskip

For every string $\,\u_{j} = (u_{j1},\dots, u_{j\ell}) \in \{0,1\}^{\ell}\,$, let $\,x_{j} \in [0,2^{\ell})\,$
be the corresponding integer in binary.
Then $u_{js}$ is true or false respectively when the $s$-th binary digit of $x_{j}$ is $1$ or $0$.
In other words, $u_{js}$ is true or false respectively when $\floor{x_{j}/2^{s-1}}$ is odd or even.
Observe that $t = \floor{x_{j}/2^{s-1}}$ is the only integer that satisfies $x_{j}/2^{s-1} - 1 < t \le x_{j}/2^{s-1}$.
Now, each term $u_{js}$ or $\lnot u_{js}$ can be expressed in $x_{j}$ as follows:
\begin{equation}\label{eq:bit}
\aligned
u_{js} &\iff \ex w \in \Z \; : \;
\begin{Bmatrix}
2w+1 &> &x_{j} / 2^{s-1} - 1\\
2w+1 &\le &x_{j} / 2^{s-1}
\end{Bmatrix},
\\
\lnot u_{js} &\iff \ex w \in \Z \; : \;
\begin{Bmatrix}
2w &> &x_{j} / 2^{s-1} - 1\\
2w &\le &x_{j} / 2^{s-1}
\end{Bmatrix}.
\\
\endaligned
\end{equation}

Let $\x = (x_{1},\dots,x_{k}) \in [0,2^{\ell})^{k}$.
Recall that each term $a_{i},b_{i},c_{i}$ in~\eqref{eq:Boo} is  $u_{js}$ or $\lnot u_{js}$ for some $j$ and $s$. So each clause $a_{i} \lor b_{i} \lor c_{i}$ can be expressed in $\x$ as:
\begin{equation}\label{eq:clause}
a_{i} \lor b_{i} \lor c_{i} \iff \ex w \in \Z \; : \; \big[ D_{i} \, (\x,w) \le \cj d_{i} \big] \lor \big[ E_{i} (\x,w) \le \cj e_{i} \big] \lor \big[ F_{i} \, (\x,w) \le \cj f_{i} \big],
\end{equation}
where three systems $\, D_{i} \, (\x,w) \le \cj d_{i},\; E_{i} (\x,w) \le \cj e_{i},\; F_{i} \, (\x,w) \le \cj f_{i} \,$ are of the form~\eqref{eq:bit} (with different $j$ and $s$ for each).
Note that the strict inequalities in~\eqref{eq:bit} can be sharpened without losing any integer solutions (see Remark~\ref{rem:sharpen}).
We define the polytopes:
\begin{equation*}
\aligned
K_{i} &= \big\{(\x,w) \in \R^{k+1}  \; : \; x_{1},\dots,x_{k},w \in [0,2^{\ell}) ,\; D_{i} \, (\x,w) \le \cj d_{i} \big\}, \\
L_{i} &= \big\{(\x,w) \in \R^{k+1}  \; : \; x_{1},\dots,x_{k},w \in [0,2^{\ell}) ,\; E_{i} \, (\x,w) \le \cj e_{i} \big\}, \\
M_{i} &= \big\{(\x,w) \in \R^{k+1}  \; : \; x_{1},\dots,x_{k},w \in [0,2^{\ell}) ,\; F_{i} \, (\x,w) \le \cj f_{i} \big\}.
\endaligned
\end{equation*}
So the RHS in~\eqref{eq:clause} can be rewritten as:
\begin{equation*}
\ex w \in \Z \;:\; (\x,w) \in K_{i} \cup L_{i} \cup M_{i}.
\end{equation*}
Let $I_{1} = I_{2} = \dots = I_{k} = [0,2^{\ell}) \cap \Z$, we see that~\eqref{eq:QSAT} is equivalent to:
\begin{equation}\label{eq:QSAT_step_2}
Q_{1} \, x_{1} \in I_{1} \quad  \dots \quad Q_{k} \, x_{k} \in I_{k} \quad : \quad  \bigwedge_{i=1}^{N} \Big[\ex w \in \Z \;:\; (\x,w) \in K_{i} \cup L_{i} \cup M_{i} \Big].
\end{equation}
For each $i$, we apply Lemma~\ref{lem:compress} (with $n=k+1,\, r=3$) to the polytopes $K_{i}, L_{i}, M_{i} \subset \R^{k+1}$.
This gives us another polytope $G_{i} \subset \R^{k+3}$ that satisfies:
\begin{equation*}
(\x,w) \in K_{i} \cup L_{i} \cup M_{i}  \quad \iff \quad \ex \v \in \Z^{2} \;:\; (\x,w,\v) \in G_{i}.
\end{equation*}
Substituting this into~\eqref{eq:QSAT_step_2}, we have an equivalent sentence:
\begin{equation}\label{eq:QSAT_step_3}
Q_{1} \, x_{1} \in I_{1} \quad  \dots \quad Q_{k} \, x_{k} \in I_{k} \, \quad : \quad  \bigwedge_{i=1}^{N} \Big[\ex \w \in \Z^{3} \;:\; (\x,\w) \in G_{i} \Big],
\end{equation}
where $\w = (w,\v) \in \Z^{3}$, and each $G_{i} \subset \R^{k+3}$.

\medskip

Notice that apart from the outer quantifiers,~\eqref{eq:QSAT_step_3} is a direct analogue of~\eqref{eq:all_P_i}, with $G_{i}$ playing the role of $P_{i}$ and $(\x,\w)$ in place of $(x,w)$.
The proof now proceeds similarly to the rest of Section~\ref{sec:main_1_proof} after~\eqref{eq:all_P_i}.
Along the proof, we need to define $G'_{i}$ and $G$ in similar manners to~\eqref{eq:P'_i_def} and~\eqref{eq:P_def}.
The variable $\y \in \Z^{2}$ is again needed to define $G'_{i}$.
$\Fib$ and $J$ from~\eqref{eq:Fib_J_def} are reused without change.
This gives us $G'_{i}, G \subset \R^{k+5}$.
At the end of the proof, we also need to apply Lemma~\ref{lem:compress} one more time to produce a single polytope $U$, just like in~\eqref{eq:apply_lem}.
The dimension $4$ in~\eqref{eq:apply_lem} is now $k+5$.
As a result, the final polytope $U$ has dimension $k+7$.
In the final form~\eqref{eq:QSAT_converted}, we will have $\x \in \Z^{k}, \y \in \Z^{2}$ and $\z = (\w,\t) \in \Z^{5}$.

\medskip

We have converted~\eqref{eq:QSAT} to an equivalent sentence~\eqref{eq:QSAT_converted} with polynomial size.
This shows that~\eqref{eq:QSAT_converted} is $\SigmaP_{k}$/$\PiP_{k}$-hard depending when $Q_{1} = \ex/\for$.
For each tuple $\x = (x_{1},\dots,x_{k})$, we can check in polynomial time whether $\for \y \in J \;\; \ex \z \in K \;:\; A \, (\x,\y,\z) \le \cj b \;$ by applying Theorem~\ref{th:Kannan}.
This shows the membership of~\eqref{eq:QSAT_converted} in $\SigmaP_{k}$/$\PiP_{k}$.
We conclude that~\eqref{eq:QSAT_converted} is $\SigmaP_{k}$/$\PiP_{k}$-complete when $Q_{1} = \ex/\for$. \ $\sq$

\bigskip

\bigskip

\section{Proof of Theorem~\ref{th:main_3}}\label{sec:main_3_proof}

\subsection{}
Now we prove Theorem~\ref{th:main_3}.
We use the same construction as in the proof of Theorem~\ref{th:main_1}.
Recall the definition of $\qrem{x\aalpha}$ from Section~\ref{sec:main_1_proof}.
We reduce the following counting problem to a problem of the form~\eqref{eq:main_3}:

\countingdef
{\#GOOD SIMULTANEOUS APPROXIMATIONS (\#GSA)}
{A rational vector $\aalpha = (\al_{1}, \dots, \al_{d}) \in \Q^{d}$ and $N \in \NN$, $\eps \in \Q$.}
{The number of integers $x \in [1,N]$ that satisfy $\qrem{x\aalpha} \le \eps\ts$.}

The argument in~\cite{Lag} is based on a parsimonious reduction. Namely,
it gives a bijection between solutions for $\#\textsc{GSA}$ and the following problem:

\countingdef
{\#WEAK PARTITIONS}
{An integer vector $\a = (a_{1},\dots,a_{d}) \in \Z^{d}$.}
{The number of $\y \in \{-1,0,1\}^{d}$ for which $\a \cdot \y = 0$.}

It is well known and easy to see that \textsc{\#WEAK PARTITIONS} is $\sharpP$-complete.
The decision version $\textsc{WEAK PARTITION}$ was earlier shown by~\cite{Boas}
to be $\NP$-complete with a parsimonious reduction from $\textsc{KNAPSACK}$.
Together with Lagarias's reduction, we conclude:

\begin{theo}
$\textsc{\#GSA}$ is $\sharpP$-complete.
\end{theo}

\subsection{}
Now we proceed with the reduction of $\textsc{\#GSA}$ to~\eqref{eq:main_3}.

\medskip

Just like the decision version, $\textsc{\#GSA}$ is only non-trivial for $\eps < 1/2$.
Define:
\begin{equation}\label{eq:Q_i_def}
Q_{i} = \big\{ (x,w) \in \R^{2} \quad : \quad 1 \le x \le N,\quad  \al_{i} x + \eps < w < \al_{i} x - \eps + 1 \big\}.
\end{equation}
Let $I = [1,N] \cap \Z$. We have:

\begin{obs}\label{obs:many_complement}
An $x \in I$ satisfies $\qrem{x\aalpha} \le \eps\ts$ if and only if for every $1 \le i \le d$, there is no  $w \in \zz$  such that $(x,w) \in Q_{i}$.
\end{obs}

Indeed, consider $x \in I$.
By~\eqref{eq:GSAIP}, we have $\qrem{x\aalpha} \le \eps$ if and only if for each $i$, there exists $w_{i} \in \zz$ with $w_{i} \in [\al_{i}x - \eps, \al_{i}x + \eps]$.
This interval of length $2\eps$ is contained in $[\al_{i}x-\eps,\al_{i}x-\eps+1)$.
The latter is a half-open unit interval, which always contains a unique integer $w_{i}$.
So $w_{i} \in [\al_{i}x - \eps, \al_{i}x + \eps]$ if and only if $w_{i} \notin (\al_{i}x + \eps, \al_{i}x - \eps + 1)$.
In other words, for each $1 \le i \le d$, there should be no $w \in \zz$ with $(x,w) \in Q_{i}$.
The converse is also straightforward.

\begin{rem}\label{rem:sharpen}
Note that each $Q_{i}$ has two open edges.
They can actually sharpened without affecting the integer points in $Q_{i}$.
Indeed, we can multiply each inequality with the denominators in $\al_{i}$ and~$\eps$, which have polynomial length.
Each resulting inequality is of the form $a < b$, with $a$ and $b$ having integer values.
This is equivalent to $a \le b - 1$.
Therefore, we can replace $Q_i$ with a (smaller) closed parallelogram containing the same integer points.
\end{rem}


By the above observation, $\textsc{\#GSA}$ asks for:
\begin{equation}\label{eq:sharpGSA_1}
N \; - \; \# \Big\{ x \in I \; : \; \ex 1 \le i \le d \quad \ex w \in \zz \quad (x,w) \in Q_{i} \Big\}.
\end{equation}
We convert the union of $Q_{i}$ into a complement $V \cpl U$ of two polytopes $U,V \subset \rr^{3}$.

\subsection{}
Let $T = 1 + N\max_{i} \al_{i}$.
Pick $d$ integers $0 < m_1 < m_2 < \dots < m_d$ so that
\begin{equation}\label{eq:m_prop}
\frac{m_{i-1} + m_{i+1}}{2} + 2T < m_{i} \quad \text{for} \quad 2 \le i \le d-1.
\end{equation}
We embed each parallelogram $Q_{i}$ into $\rr^{3}$ as
\begin{equation}\label{eq:translate}
R_{i} = \big\{ (x,y,w) \in \rr^{3} \; : \; (x,w-m_{i}) \in Q_{i},\; y = i \big\}.
\end{equation}
In other words, we translate $Q_{i}$ by $m_{i}$ in the direction $w$, and embed it into the plane $y = i$ inside $\rr^{3}$ (see Figure~\ref{f:Ri}).
The following is obvious:

\begin{obs}\label{obs:translated_complement}
For each $x \in I$ and $1 \le i \le d$, there exists $w \in \zz$ with $(x,w) \in Q_{i}$ if and only if there exists $(y,w) \in \zz^{2}$ with $(x,y,w) \in R_{i}$.
\end{obs}


Denote by $A_{i},B_{i},C_{i}$ and $D_{i}$ the vertices of $R_{i}$.
Let $K_{i} = (N,i,0)$ and $L_{i} = (1,i,0)$ for each $1 \le i \le d$.
Define:
\begin{equation}\label{eq:UV_def}
\aligned
U &= \conv \big\{A_{i},\ts B_{i},\ts K_{i},\ts L_{i} \,:\, 1 \le i \le d \big\} \subset \rr^{3},\\
V &= \conv \big\{C_{i},\ts D_{i},\ts K_{i},\ts L_{i} \,:\, 1 \le i \le d \big\} \subset \rr^{3}.
\endaligned
\end{equation}

\begin{figure}[hbt]
\begin{center}
\psfrag{Ki}{$K_i$}
\psfrag{Li}{$L_i$}
\psfrag{Ai}{$A_i$}
\psfrag{Bi}{$B_i$}
\psfrag{Ci}{$C_i$}
\psfrag{Di}{$D_i$}
\psfrag{Ri}{$R_i$}
\psfrag{x}{$x$}
\psfrag{y}{$y$}
\psfrag{w}{$w$}
\psfrag{R1}{$R_1$}
\psfrag{R2}{$R_2$}
\psfrag{Rd}{$R_d$}

\epsfig{file=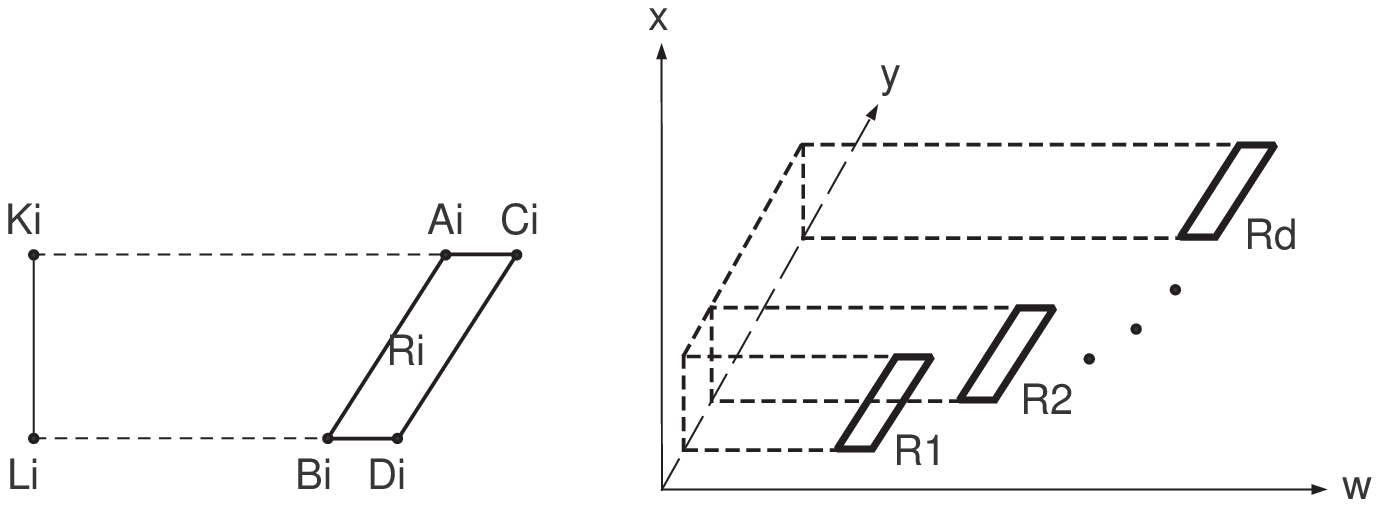, width=12cm}
\vskip-.15cm
\end{center}
\caption{The parallelograms $R_{i}$.}
\label{f:Ri}
\end{figure}

Since $\conv(A_{i},B_{i},K_{i},L_{i}) \subset \conv(C_{i},D_{i},K_{i},L_{i})$ for each $1 \le i \le d$, we have $U \subset V$.
It is also clear that:
\begin{equation}\label{eq:R_i_complement}
R_{i} = \conv(C_{i},D_{i},K_{i},L_{i}) \cpl \conv(A_{i},B_{i},K_{i},L_{i}).
\end{equation}
Denote by $\{y=i\}$ the plane $y=i$.

\begin{obs}\label{obs:y_i}
We have $U \cap \{y=i\} = \conv(A_{i},B_{i},K_{i},L_{i})$.
Similarly, $V \cap \{y=i\} = \conv(C_{i},D_{i},K_{i},L_{i})$.
\end{obs}

Indeed, from~\eqref{eq:UV_def}, it is clear that $\conv(A_{i},B_{i},K_{i},L_{i})$ lies in both $U$ and the plane $y=i$.
On the other hand, if $(x,i,w) \in U$, it must be a convex combination of $A_{j},B_{j},K_{j},L_{j}$ for $1 \le j \le d$.
First, assume that
\begin{equation}\label{eq:conv_other}
(x,i,w) \in \conv\big\{A_{j},B_{j},K_{j},L_{j} : j \neq i \big\}.
\end{equation}
From~\eqref{eq:Q_i_def} and~\eqref{eq:translate}, the $w$-coordinates of $A_{j},B_{j},C_j,D_{j}$ are within the range $[m_{j}, m_{j}+T]$.
For $K_{j}$ and $L_{j}$, their $w$-coordinates are $0$.
Therefore, by the convexity condition~\eqref{eq:m_prop}, any point $(x,y,w)$ as in~\eqref{eq:conv_other} must have $w < m_{i} - T < m_{i}$.
This implies that $(x,i,w) \in \conv\big\{A_{i},B_{i},K_{i},L_{i}\big\}$, because the $w$-coordinates of $A_{i}$ and $B_{j}$ are at least $m_{i}$.
So we have
\begin{equation*}
\conv\big\{A_{j},B_{j},K_{j},L_{j} : j \neq i \big\} \cap \{y=i\} \subset \conv\big\{A_{i},B_{i},K_{i},L_{i}\big\}.
\end{equation*}
Adding $A_{i},B_{i},C_{i}$ and $D_{i}$ to the LHS, we have
\begin{equation*}
\conv\big\{A_{j},B_{j},K_{j},L_{j} : 1 \le j \le d \big\} \cap \{y=i\} = \conv\big\{A_{i},B_{i},K_{i},L_{i}\big\}.
\end{equation*}
This proves the observation for $U$.
The same argument works for $V$.


By Observation~\ref{obs:y_i}, for $(x,y,w) \in \zz^{3}$, we have $(x,y,z) \in V \cpl U$ if and only if
\begin{equation*}
(x,y,w) \in \conv(C_{i},D_{i},K_{i},L_{i}) \cpl \conv(A_{i},B_{i},K_{i},L_{i})
\end{equation*}
for some $1 \le i \le d$.
Combined with~\eqref{eq:R_i_complement} and Observation~\ref{obs:translated_complement}, for every $x \in I$, we have:
\begin{equation*}
\ex (y,w) \in \zz^{2} \quad (x,y,w) \in V \cpl Q \quad \iff \quad \ex 1 \le i \le d \quad \ex w \in \zz \quad (x,w) \in Q_{i}.
\end{equation*}
From~\eqref{eq:sharpGSA_1}, we conclude that $\textsc{\#GSA}$ is exactly:
\begin{equation*}
N \; - \; \# \Big\{ x \in I \; : \; \ex (y,z) \in \zz^{2} \quad (x,y,w) \in V \cpl U \Big\} = N - |\rE_{1}(V \cpl U)|.
\end{equation*}
Let $P = U, Q = V$ we have Theorem~\ref{th:main_3}.

\subsection{Proof of Corollary~\ref{th:many_proj}}
By Theorem~\ref{th:main_3}, counting $|\rE_{1}(Q \cpl P)|$ is $\sharpP$-complete for $P \subset Q \subset \R^{3}$.
Nevertheless, the complement $Q \cpl P$ can still be triangulated into polynomially many simplices $T_{1} \sqcup \dots \sqcup T_{r}$.
In fact, by an application of Proposition 5.2.2 in~\cite{W}, the systems describing all such $T_{i}$ can be found in polynomial time.
Therefore, counting $|\rE_{1}(T_{1} \sqcup \dots \sqcup T_{r})| = |\rE_{1}(Q \cpl P)|$ is $\sharpP$-complete. \ $\sq$

\bigskip

\section{Another hard decision problem}\label{sec:2_quant}
Our construction with Fibonacci points also yields the following completeness result with only $2$ quantifiers:

\begin{theo}\label{th:2_quant}
Given three polytopes $U_{1},U_{2},U_{3} \subset \R^{4}$ and two boxes $I \subset \zz, K \subset \zz^{3}$, deciding the sentence:
\begin{equation}\label{eq:2_quant}
\ex x \in I \quad \for \z \in K \quad : \quad (x,\z) \in U_{1} \cup U_{2} \cup U_{3}
\end{equation}
is $\NP$-complete.
\end{theo}

Here the condition $(x,\z) \in U_{1} \cup U_{2} \cup U_{3}$ is expressed as a disjunction of three systems in four variables $x,z_{1},z_{2},z_{3}$.
Instead of many as in~\eqref{eq:Schoning}, we only need only $2$ disjunctions to express $(x,\z) \in U_{1} \cup U_{2} \cup U_{3}$.
Also notice that the quantifiers are $\ex\for$ as opposed to $\for\ex$ in Theorem~\ref{th:Kannan}.

\begin{proof}[Sketch of proof]
We again find a reduction of $\textsc{GSA}$.
Let $T = 1 + N\max_{i} \al_{i}$.
Recall $P_{i}$ from~\eqref{eq:P_i_def}.
For every $1 \le i \le d$, define two new polygons:
\begin{equation*}
\aligned
L_{i} &= \{(x,w) \in \R^{2} \; : \; 1 \le x \le N,\;  -1 \le w \le  \al_{i} x + \eps - 1\},\\
M_{i} &= \{(x,w) \in \R^{2} \; : \; 1 \le x \le N,\;  \al_{i} x - \eps \le w \le T\}.
\endaligned
\end{equation*}

\begin{obs}\label{obs:simple}
For every $x \in [1,N]$ and $1 \le i \le d$,
we have:
\begin{equation}\label{eq:ex_to_for}
\ex w \in \zz \;:\; (x,w) \in P_{i} \quad \iff \quad \for w \in [-1,T] \cap \zz \;:\; (x,w) \in L_{i} \cup M_{i}.
\end{equation}
\end{obs}

Indeed, by~\eqref{eq:P_i_def}, we have $\ex w \in \zz : (x,w) \in P_{i}$ if and only if $[\al_{i}x - \eps, \al_{i}x + \eps]$ contains an integer point $w$.
Also notice that $[\al_{i}x - \eps, \al_{i}x + \eps] \subset (\al_{i}x + \eps - 1, \al_{i}x + \eps]$ and
\begin{equation*}
[-1,T] \;=\; [-1,\al_{i}x + \eps - 1] \; \sqcup \; (\al_{i}x + \eps - 1, \al_{i}x + \eps] \;\sqcup\; (\al_{i}x + \eps, T].
\end{equation*}
Since $(\al_{i}x + \eps - 1, \al_{i}x + \eps]$ is a half-open unit interval, it contains a unique integer point $w$.
So $w$ lies in $[\al_{i}x - \eps, \al_{i}x + \eps]$ if and only if
\begin{align*}
[-1,T] \cap \zz \; &= \; \big( [-1,\al_{i}x + \eps - 1] \; \sqcup \; [\al_{i}x - \eps, \al_{i}x + \eps] \; \sqcup \; (\al_{i}x + \eps,T] \big) \cap \zz \\
&= \; \big( [-1,\al_{i}x + \eps - 1] \; \sqcup \; [\al_{i}x - \eps, T] \big) \cap \zz.
\end{align*}
This last condition is exactly the RHS in~\eqref{eq:ex_to_for}.

\medskip


Recall the Fibonacci points $\Fib = \{\fib_{1},\dots,\fib_{d}\}$.
We construct $L'_{i}, M'_{i}$ similarly to~\eqref{eq:P'_i_def} and $L,M$ similarly to~\eqref{eq:P_def} using the same Fibonacci points.
As a direct analogy to~\eqref{eq:intermediate_2}, $\textsc{GSA}$ is equivalent to:
\begin{equation}\label{eq:2_quant_intermediate}
\ex x \in I \quad \for \y \in \Fib \quad \for w \in [-1,T] \cap \zz \quad : \quad (x,\y,w) \in L \cup M.
\end{equation}

Recall $J$ from~\eqref{eq:Fib_J_def}.
Let $K = J \times \big( [-1,T] \cap \zz \big)$, which is a box in $\zz^{3}$.
Let $\z = (\y,w) \in K$.
Also recall $R_{1}$ and $R_{2}$ from~\eqref{eq:part_above} and~\eqref{eq:part_below}.
Define
\begin{equation*}
U_{1} = [1,N] \times R_{1} \times [-1,T], \quad U_{2} = \conv \big( [1,N] \times R_{2} \times [-1,T] \ts , \; L \big)
, \quad U_{3} = M.
\end{equation*}
From properties (F3)--(F5), it is not hard to see that~\eqref{eq:2_quant_intermediate} is equivalent to:
\begin{equation*}
\ex x \in I \quad \for \z \in K \quad : \quad (x,\z) \in U_{1} \cup U_{2} \cup U_{3}.
\end{equation*}
\end{proof}




\bigskip

\section{Final remarks and open problems} \label{s:finrem}

\subsection{}  \label{ss:finrem-doignon}
It is sufficient to prove Theorem~\ref{th:Kannan} for the case when $m,n$ are also bounded.
In the system $A \, (\x,\y) \le b$, we view $\x$ as the parameters and $\y$ as the variables to be solved for.
For a fixed $d_{2}$ and $m\ge 2^{d_{2}}$, the \emph{Doignon--Bell--Scarf theorem}
\cite[$\S$16.5]{Schrijver} implies that the system $A \, (\x,\y) \le \b$
is solvable in $\y \in \Z^{d_{2}}$ if and only if every
subsystem $A' \, (\x,\y) \le \bprime$ is solvable.
Here $A'$ is a submatrix with $2^{d_{2}}$ rows from $A$ with $\bprime$ the corresponding subvector from~$\b$.
In other words:
\[
\ex \y \in \Z^{d_{2}} \quad A \, (\x,\y) \le \b \quad \iff \quad \bigwedge_{(A',\, \bprime)} \Big[ \ex \y \in \Z^{d_{2}} \quad A' \, (\x,\y) \le \bprime \Big].
\]
The total number of pairs $(A',\, \bprime)$ is $\displaystyle\binom{m}{2^{d_{2}}}$, which is polynomial in $m$.

Note that the conjunction over all $(A',\, \bprime)$ commutes with the universal quantifier $\;\for \x$.
Therefore:
\begin{equation*}
\for \x \in P \cap \Z^{d_{1}} \;\; \ex \y \in \Z^{d_{2}} \;\; A \, (\x,\y) \le \b \;
\iff
\bigwedge_{(A',\, \bprime)} \Big[ \for \x \in P \cap \Z^{d_{1}} \;\; \ex \y \in \Z^{d_{2}} \;\; A' \, (\x,\y) \le \bprime \Big].
\end{equation*}
Thus, it is equivalent to check each of the smaller subproblems, each of which has $m = 2^{d_{2}}$.
Recall that the number of facets in $P$ is $n$, which can still be large.
However, given the system $C \, \x \le \cj\ga$ describing $P$, we can triangulate $P$ into to a union of simplices $P_{1} \sqcup \dots \sqcup P_{k}$.
Since the dimension $d_{1}$ is bounded,
we can find such a triangulation in polynomial time (see e.g.~\cite{DRS}).
Now for each pair $(A',\, \bprime)$, we have:
\begin{equation*}
\for \x \in P \cap \Z^{d_{1}} \;\; \ex \y \in \Z^{d_{2}} \;\; A' \, (\x,\y) \le \bprime \;
\iff
\bigwedge_{i=1}^{k} \Big[ \for \x \in P_{i} \cap \Z^{d_{1}} \;\; \ex \y \in \Z^{d_{2}} \;\; A' \, (\x,\y) \le \bprime \Big].
\end{equation*}
Each simplex $P_{i} \subset \R^{d_{1}}$ has $d_{1}+1$ facets.
Each subsentence in the RHS now has $m=2^{d_{2}}$ and $d_{1}+1$.
Note that the total number of such subsentences is still polynomial,
so it suffices to check each of them individually.

\smallskip

For three quantifiers $\ex \x \; \for \y \; \ex \z$, this argument breaks down because the existential quantifier $\ex \x$ no longer commutes with a long conjunction.

\subsection{} \label{ss:BP}
By taking finite Boolean combinations, we see that Theorem~\ref{th:B}
also allows counting  integer points in a union of $k$ polytopes,
where $k$ is bounded (see~\cite{B3,BP}).  In fact, Woods proved in~\cite[Prop.\ 5.3.1]{W}
that it is still possible to count all such points in polynomial time when $k$
is arbitrary.
By Corollary~\ref{th:many_proj}, we see that this is not the case for projection.

\subsection{} \label{ss:GSA}
The \textsc{GSA} Problem plays an important role in both Number Theory and Integer Programming
especially in connection to lattice reduction algorithms (see e.g.~\cite{GLS}). Let us mention
that via a chain of parsimonious reductions one can show that \textsc{\#GSA} is also hard to
approximate (cf.~\cite{ER}).  Note also that \textsc{GSA} has been recently used in a somewhat related
geometric context in~\cite{EH}.

\subsection{} \label{ss:compression}
An easy consequence of Lemma~\ref{lem:compress} proves the first part
of the following result:

\begin{prop} \label{p:log}
Every set $S = \{\cj p_{1},\dots, \cj p_{r}\} \subset \zz^{2}$ is a
projection of integer points of some convex polytope $P \subset \rr^{2+d}$,
where \ts$d\le \lceil \log_{2} r \rceil$.  Moreover, the bound \ts
$d\le \lceil \log_{2} r \rceil$ \ts is tight.
\end{prop}

We only use the proposition to reduce the dimension of variable $\bz$
in Theorem~\ref{th:main_1_restated} from $4$ to~$3$, but it is perhaps of independent interest.
Note that a weaker inequality $d\le r$ is trivial.

\begin{proof}[Proof of the second part of Proposition~\ref{p:log}]
Consider a set $S=\{\cj p_{1},\dots,\cj p_{r}\}$ of integer points in convex position
and with even coordinates.
Assume there is a polytope $P \subset \rr^{2+\ell}$ with $\ell < \lceil \log_{2} r \rceil$ so that $S$ is exactly the projection of $P \cap \zz^{2+\ell}$ on $\zz^{2}$.
Then there are integer points $\cj q_{1},\dots,\cj q_{r} \in \zz^{\ell}$ so that $(\cj p_{i}, \cj q_{i}) \in P$.
Since $r > 2^{\ell}$, by the pigeonhole principle, we have $\cj q_{i} - \cj q_{j} \in 2 \zz^{\ell}$ for some $i \neq j$.
Then the midpoint of  $(\cj p_{i}, \cj q_{i})$ and $(\cj p_{j}, \cj q_{j})$ is an integer point in $\zz^{2+\ell}$,
which also lies in $P$ by convexity.
The projection of this midpoint on $\zz^{2}$ is $(\cj p_{i} + \cj p_{j})/2$, which must lie in~$S$.
However, the points in $S$ are in convex positions and thus contain no midpoints, a contradiction.
\end{proof}


\subsection{}  \label{ss:finrem-diff}
Let us give another motivation behind Theorem~\ref{th:main_3} and put it into context
of our other work.  In this paper, we bypass the ``short generating function'' technology
developed for computing $|\rE_1(P)|$ for convex polytopes $P\ssu \rr^d$.  Note, however, that for $X=Q \cpl P$ as in
the theorem, the corresponding short GF $f_X(\t)$ is simply the difference $f_Q(\t)-f_P(\t)$, which can still be computed in polynomial time (see~\cite{B1}).
Thus,
if one could efficiently present the projection of $f_X(\t)$ on $\Z$ as a short generating function of polynomial size, then one would be able to compute
$|\rE_1(Q\cpl P)|$, a contradiction.  In other words, Theorem~\ref{th:main_3} is an extension of a result by Woods~\cite{Woods}, which shows that computing projecting short generating functions is $\NP$-hard.  It is also an effective but weaker version of the main
result in~\cite[Th.~1.3]{squares}, which deals with the size of short GFs of the projections rather than
complexity of their computation.

\subsection{}  \label{ss:many_proj}
Corollary~\ref{th:many_proj} says that computing $|\rE_{1}(T_{1} \cup\dots\cup T_{k})|$ is $\sharpP$-complete
even for simplices $T_{i} \subset \rr^{3}$.
By a stronger version of Theorem~\ref{th:BW} (see~\cite{BW}), for each polytope~$T_{i}$,
there is a short generating function $g_{i}(t)$ representing $\rE_{1}(T_{i})$.
The union of all those generating functions correspond to $\rE_{1}(T_{1} \cup\dots\cup T_{k})$.
As a corollary we conclude that the union operation on short generating functions is $\sharpP$-hard to compute.
As in $\S$\ref{ss:finrem-diff} above, one should compare this to a stronger result~\cite[Th.~1.1]{squares},
which says that the union of short generating functions can actually have super-polynomial lengths unless \ts
$\SP \subseteq \FPPo$.

\subsection{}  \label{ss:finrem-diff-more}
Dimension $3$ in Theorem~\ref{th:main_3} is optimal.
Indeed, assume $P,Q \subset \rr^{2}$.
Then one can decompose $Q \cpl P = R_{1} \cup \dots \cup R_{r}$, where each $R_{i}$ is a polygon, so that the projection $\rE_{1}(R_{i})$ onto the $x$-axis of each $R_{i}$ intersects at most one other $\rE_{1}(R_{j})$.
This can easily be done by drawing vertical lines through vertices of $P$, which together with $\partial P$ will divide $Q \cpl P$ into $R_{1},\dots,R_{r}$.
By Theorem~\cite{BW}, we can find a generating function $g_{i}(t)$ for each $\rE_{1}(R_{i})$ in polynomial time.
From Corollary 3.7 in~\cite{BW}, the union $g(t)$ of all $g_{i}(t)$ can also be found in polynomial time, because each of them intersects at most one another in support.
Evaluating $g(1)$, we get the count for $|\rE_{1}(Q \cpl P)|$.



\subsection{} \label{ss:finrem-dim}
Note that Theorem~\ref{th:main_1_restated} was proved for dimensions $d_{1} = 1, d_{2} = 2$ and $d_{3} = 3$.
One can ask if the problem still remains $\NP$-complete when some of these dimensions are lowered.
In particular, it would be interesting to see if the following problem is still $\NP$-complete:
\begin{equation*}
\ex x \in P \cap \zz \quad \for \y \in Q \cap \zz^2 \quad \ex \z \in \Z^{2} \quad : \quad (x,\y,\z) \in U,
\end{equation*}
where $P \subset \R,\, Q \subset \R^{2}$ and $U \subset \R^{5}$ are convex polytopes.

\vskip.56cm

\subsection*{Acknowledgements}
We are grateful to Iskander Aliev, Matthias Aschenbrenner,
Sasha Barvinok, Matt Beck, Art\"{e}m Chernikov, Jes\'{u}s De Loera,
Matthias K\"oppe, Sinai Robins and Kevin Woods for interesting conversations and helpful remarks.
The second author was partially supported by the~NSF.


\end{document}